\documentclass[12pt]{amsart}
\usepackage{latexsym,fancyhdr,amssymb,color,amsmath,amsthm,graphicx,listings,comment}
\usepackage[section]{placeins}
\newtheorem{thm}{Theorem} \newtheorem{lemma}{Lemma}  \newtheorem{coro}{Corollary}
\newtheorem{propo}{Proposition} 
\definecolor{red1}{rgb}{1,0.9,0.9} \definecolor{blue1}{rgb}{0.9,0.9,1} \definecolor{green1}{rgb}{0.9,1,0.9} 
\definecolor{yellow1}{rgb}{1,1,0.9} \definecolor{yellow2}{rgb}{1,1,0.8}

\let\paragraph\subsection
\newcommand{\str}{\rm{str}}

\newcommand{\ZZ}{\mathbb{Z}}
\newcommand{\RR}{\mathbb{R}}
\newcommand{\CC}{\mathbb{C}}
\newcommand{\G}{\mathcal{G}}

\newcommand{\A}{\mathcal{A}}
\def\osquare{\mbox{\ooalign{$\times$\cr\hidewidth$\square$\hidewidth\cr}} }

\title{The energy of a simplicial complex}
\author{Oliver Knill}
\date{July 7, 2019, [This document contains already announced results but is a fresh write-up with sometimes more detailed proofs.]}
\address{Department of Mathematics \\ Harvard University \\ Cambridge, MA, 02138 }


\begin{document}
\maketitle

\begin{abstract}
A finite abstract simplicial complex $G$ defines
a matrix $L$, where $L(x,y)=1$ if two simplicies $x,y$
in $G$ intersect and where $L(x,y)=0$ if they don't. This matrix is always unimodular
so that the inverse $g=L^{-1}$ has integer entries $g(x,y)$. In analogy to
Laplacians on Euclidean spaces, these Green function entries define 
a potential energy between
two simplices $x,y$. We prove that the total energy $E(G) = \sum_{x,y} g(x,y)$ 
is equal to the Euler characteristic $\chi(G)$ of $G$ and that the 
number of positive minus the number of negative eigenvalues of $L$ is equal to $\chi(G)$.
\end{abstract}

\section{The theorems}

\paragraph{}
A finite set $G$ of non-empty sets which is closed under the operation
of taking finite non-empty subsets is called a {\bf finite abstract 
simplicial complex}. The $n$ elements in G are called {\bf simplices} or 
{\bf faces}, the $n \times n$ matrix $L$ satisfying $L(x,y)=1$ if $x$ and 
$y$ intersect and $L(x,y)=0$ else is the {\bf connection matrix} of $G$. 
Define ${\rm dim}(x)=|x|-1$, where $|x|$ is the cardinality of $x$. 
If $\omega(x) =(-1)^{{\rm dim}(x)}$, then $\chi(G)=\sum_{x \in G} \omega(x)$ is
the {\bf Euler characteristic} of $G$. A multiplicative analog of $\chi(G)$ is
the {\bf Fermi characteristic} $\phi(G) = \prod_{x \in G} \omega(x) \in \{-1,1\}$.

\paragraph{}
\begin{thm}[Unimodularity theorem \cite{Unimodularity} (2016)] ${\rm det}(L) = \phi(G)$.  
\label{1}
\end{thm}

\paragraph{}
It follows from the Cramer formula in linear algebra
that the inverse matrix $g=L^{-1}$ has integer entries
$g(x,y)$. We can think of $g(x,y)$ as the {\bf potential energy} between the 
simplices $x$ and $y$. The number $E(G)=\sum_{x,y \in G} g(x,y)$ is the
{\bf total energy} of $G$. 

\begin{thm}[Energy theorem \cite{Helmholtz} (2017) ] $E(G)=\chi(G)$.
\label{2}
\end{thm}

\paragraph{}
Let $p(G)$ be the number of positive eigenvalues of $L$ and $n(G)$
the number of negative eigenvalues of $L$. Let $b(G)$ denote the
number of even dimensional simplices in $G$ and $f(G)$ the number of odd
dimensional simplices in $G$. If $f_k(G)$ is the number of elements in $G$
with cardinality $k+1$, then $(f_0,f_1, \cdots, f_d)$ is the {\bf f-vector}
of $G$ and $\chi(G)=\sum_{k \; {\rm even}} f_k - \sum_{k \; {\rm odd}} f_k
=b(G)-f(G)$. 

\begin{thm}[Hearing Euler characteristic \cite{ListeningCohomology}, (2018)]
We have $b(G)=p(G)$ and $f(G)=n(G)$. Therefore, $\chi(G)=p(G)-n(G)$. 
\end{thm}

\paragraph{}
It follows from $|{\rm det}(L)|=1$ that ${\rm tr}(\log(|L|))=0$ so that
$$  \chi(G) = \frac{2}{i \pi} {\rm tr}(\log(i L))  $$
if the branch ${\rm arg}(\log(z)) \in [0,2\pi)$ is chosen. 
This writes $\chi(G)$ as a {\bf logarithmic energy} of a spectral set 
$\sigma(i L)$ on the imaginary axes of the complex plane again reinforcing
that $\chi(G)$ is a type of energy. 

\paragraph{}
Let $W^+(x) = \{ y \in G, x \subset y \}$ denote the {\bf star} of $x$.
It is a set of sets and not necessarily a simplicial complex but still
has an Euler characteristic. The Green function entries are explicitly 
known:

\begin{thm}[Green star formula \cite{Helmholtz} (2017)]
The Green function entries are
$g(x,y) = \omega(x) \omega(y) \chi(W^+(x) \cap W^+(y))$. 
\end{thm}

\paragraph{}
The potential energies are {\bf local}, of bounded range, unlike in Euclidean 
spaces where potentials are long range. Simplicial complexes have a natural
hyperbolic structure for which the star is the unstable
manifold of a gradient vector field of the dimension functional. The stable
manifold $W^-(x) = \{ y \in G, y \subset x \}$ is the simplicial complex 
generated by $x$ and $\chi(W^-(x) \cap W^-(y))=L(x,y)$. 
As $\omega(x) = \chi(W^-(x))$, both $L$ and $g$ have entries given as 
the Euler characteristic of ``homoclinic" or ``heteroclinic points" of
a hyperbolic dynamical system.

\paragraph{}
The disjoint union of complexes defines an additive monoid  which can be
completed to become a group. The group operation $G+H$ is represented by the 
direct product of matrices $L(G) \oplus L(H)$. An element in the group is naturally
described by its connection Laplacian $L$ postulating $L(-G)=-L(G)$. 
The Cartesian product of two simplicial 
complexes is not a simplicial complex but it carries a natural exterior derivative $d$
defining a Hodge Laplacian $H(G)=(d+d^*)^2$ which is a direct sum of form-Laplacians $L_k$
for which by Hodge, the nullity of the kernel are the Betti numbers $b_k={\rm ker}(L_K)$ and
for which the K\"unneth formula holds so that on the full ring $\mathcal{G}$ the Poincar\'e map 
$p(G) = b_0 t + b_1 t^2 + \cdots +b_d t^d$ is a ring homomorphism to polynomials. 
The spectrum of the Hodge Laplacian is not compatible in general with multiplication. 
It is however with the connection Laplacian: 

\begin{thm}[Tensor algebra representation \cite{StrongRing}, (2017)]
The map $G \to L(G)$ is a representation of the ring $\mathcal{G}$ in a tensor
algebra of finite dimensional invertible matrices.
\label{representation}
\end{thm}

\paragraph{}
The connection Laplacian $L(G)$ acts on the same Hilbert space than $H(G)$. 
The multiplication of complexes defines a strong product for the corresponding
connection graphs and the {\bf tensor product} of the connection Laplacians $L(G)$. 
It follows that when adding two simplicial complexes, the spectrum is the union of 
the spectra as for ``independent quantum mechanical processes". When 
multiplying two complexes, then the spectra of $L$ multiply and are mathematically
described in the same way than multi-particle states appear in physics. 

\section{Examples}

\paragraph{}
There are various ways to build simplicial complexes: 
any finite set of finite non-empty sets $A$ {\bf generates} a complex 
$G = \{ x \subset y \; | \; x \neq \emptyset, y \in A \}$.
Given a complex $G$, the {\bf $k$-skeleton} of $G$ is the set of subsets of $G$ of dimension $\leq k$. 
A finite simple graph $(V,E)$ generates the {\bf Whitney complex}
$G = \{ x \subset V \; | \; \forall a,b \in x, (a,b) \in E\}$, where the simplices are the vertex 
sets $V(K)$ of complete subgraphs $K$ of $(V,E)$. The Whitney complex is also known under the name
{\bf clique complex}. Its dual is the {\bf independence complex}, in which the simplices are 
the independent sub-sets of $V$. An other example of a complex defined on a graph is the 
{\bf graphic matroid} $G = \{ x \subset E \; | \; x$ generates a forest in $(V,E) \}$.

\paragraph{}
{\bf Example 1)} The set of sets $A=\{ (1,2,3),(2,3,4) \}$ generates
the complex $G = \{ (1),(2),(3),(4)$, $(1,4),(1,3),(2,3),(3,4)$, $(1,2,3),(2,3,4)\}$
which is the Whitney complex of the diamond graph. The $f$-vector is $(4,5,2)$, the
Euler characteristic $\chi(G)=4-5+2=1$, the Fermi characteristic $(-1)^5=-1$ which 
agrees with the determinant of $L$ and the determinant of $g=L^{-1}$
\begin{tiny}
$$ g=\left[
                  \begin{array}{ccccccccccc}
                   0 & 0 & 0 & 0 & 0 & 0 & -1 & 0 & 0 & 1 & 0 \\
                   0 & 0 & 1 & 0 & 0 & -1 & -1 & 0 & -1 & 1 & 1 \\
                   0 & 1 & 0 & 0 & -1 & 0 & -1 & -1 & 0 & 1 & 1 \\
                   0 & 0 & 0 & 0 & 0 & 0 & -1 & 0 & 0 & 0 & 1 \\
                   0 & 0 & -1 & 0 & 0 & 1 & 1 & 0 & 0 & -1 & 0 \\
                   0 & -1 & 0 & 0 & 1 & 0 & 1 & 0 & 0 & -1 & 0 \\
                   -1 & -1 & -1 & -1 & 1 & 1 & 1 & 1 & 1 & -1 & -1 \\
                   0 & 0 & -1 & 0 & 0 & 0 & 1 & 0 & 1 & 0 & -1 \\
                   0 & -1 & 0 & 0 & 0 & 0 & 1 & 1 & 0 & 0 & -1 \\
                   1 & 1 & 1 & 0 & -1 & -1 & -1 & 0 & 0 & 1 & 0 \\
                   0 & 1 & 1 & 1 & 0 & 0 & -1 & -1 & -1 & 0 & 1 \\
                  \end{array}
                  \right] \; .  $$
\end{tiny}
% G=GenerateComplex[{{1,2,3},{2,3,4}}]; FvectorComplex[G]; L=ConnectionLaplacianComplex[G]; g=Inverse[L]

\paragraph{}
{\bf Example 2)} 
Not every complex is a Whitney complex. The complex generated by 
$A=\{ \{1,2\}$, $\{2,3\}$, $\{3,1\} \}$ is 
$$  G=C_3=\{ \{1\},\{2\},\{3\},\{1,2\},\{2,3\},\{3,1\} \} \, $$
which is the $1$-skeleton complex $C_3$ of of the triangle 
$K_3$. While $K_3$ is the Whitney complex of a graph, the complex $C_3$ is not.
Its $f$-vector is $(3,3)$ with Euler characteristic $3-3=0$ and Fermi characteristic
$(-1)^3=-1$. 
% G=GenerateComplex[{{1,2},{2,3},{1,3}}]; FvectorComplex[G]; L=ConnectionLaplacianComplex[G]; g=Inverse[L]
\begin{tiny}
$$ L=\left[
                  \begin{array}{cccccc}
                   1 & 0 & 0 & 1 & 1 & 0 \\
                   0 & 1 & 0 & 1 & 0 & 1 \\
                   0 & 0 & 1 & 0 & 1 & 1 \\
                   1 & 1 & 0 & 1 & 1 & 1 \\
                   1 & 0 & 1 & 1 & 1 & 1 \\
                   0 & 1 & 1 & 1 & 1 & 1 \\
                  \end{array}
                  \right], g = \left[
                  \begin{array}{cccccc}
                   -1 & -1 & -1 & 1 & 1 & 0 \\
                   -1 & -1 & -1 & 1 & 0 & 1 \\
                   -1 & -1 & -1 & 0 & 1 & 1 \\
                   1 & 1 & 0 & -1 & 0 & 0 \\
                   1 & 0 & 1 & 0 & -1 & 0 \\
                   0 & 1 & 1 & 0 & 0 & -1 \\
                  \end{array}
                  \right] \; . $$
\end{tiny}
We check ${\rm det}(L)=-1$ and $\sum_{x,y \in G} g(x,y)=0$. 

{\bf Example 3)} If $G=\{ \{1,2\},\{1\},\{2\} \},
H=\{ \{1,2\},\{2,3\},\{1\},\{2\},\{3\} \}$ then 
$$ L(G) = \left[ \begin{array}{ccc} 1 & 1 & 1 \\ 1 & 1 & 0 \\ 1 & 0 & 1 \\ \end{array} \right],
   L(H) = \left[ \begin{array}{ccccc} 1 & 1 & 1 & 1 & 0 \\ 1 & 1 & 0 & 1 & 1 \\
           1 & 0 & 1 & 0 & 0 \\ 1 & 1 & 0 & 1 & 0 \\ 0 & 1 & 0 & 0 & 1 \\ \end{array} \right] \; , $$
with ${\rm det}(L(G))=-1$ and ${\rm det}(L(H))=1$. 
The eigenvalues of $L(G)$ are $\left\{1+\sqrt{2},1,1-\sqrt{2}\right\}$, the eigenvalues of $L(H)$
are 
$$ \sigma(L(H))=\left\{\frac{1}{2} \left(3+\sqrt{13}\right),\frac{1}{2}
   \left(1+\sqrt{5}\right),1,\frac{1}{2} \left(1-\sqrt{5}\right),\frac{1}{2}
   \left(3-\sqrt{13}\right)\right\} \; . $$
The product complex $G \times H$ is the tensor product of $L(G)$ and $L(H)$ which 
\begin{tiny}
$$ \left[
\begin{array}{ccccccccccccccc}
 1 & 1 & 1 & 1 & 0 & 1 & 1 & 1 & 1 & 0 & 1 & 1 & 1 & 1 & 0 \\
 1 & 1 & 0 & 1 & 1 & 1 & 1 & 0 & 1 & 1 & 1 & 1 & 0 & 1 & 1 \\
 1 & 0 & 1 & 0 & 0 & 1 & 0 & 1 & 0 & 0 & 1 & 0 & 1 & 0 & 0 \\
 1 & 1 & 0 & 1 & 0 & 1 & 1 & 0 & 1 & 0 & 1 & 1 & 0 & 1 & 0 \\
 0 & 1 & 0 & 0 & 1 & 0 & 1 & 0 & 0 & 1 & 0 & 1 & 0 & 0 & 1 \\
 1 & 1 & 1 & 1 & 0 & 1 & 1 & 1 & 1 & 0 & 0 & 0 & 0 & 0 & 0 \\
 1 & 1 & 0 & 1 & 1 & 1 & 1 & 0 & 1 & 1 & 0 & 0 & 0 & 0 & 0 \\
 1 & 0 & 1 & 0 & 0 & 1 & 0 & 1 & 0 & 0 & 0 & 0 & 0 & 0 & 0 \\
 1 & 1 & 0 & 1 & 0 & 1 & 1 & 0 & 1 & 0 & 0 & 0 & 0 & 0 & 0 \\
 0 & 1 & 0 & 0 & 1 & 0 & 1 & 0 & 0 & 1 & 0 & 0 & 0 & 0 & 0 \\
 1 & 1 & 1 & 1 & 0 & 0 & 0 & 0 & 0 & 0 & 1 & 1 & 1 & 1 & 0 \\
 1 & 1 & 0 & 1 & 1 & 0 & 0 & 0 & 0 & 0 & 1 & 1 & 0 & 1 & 1 \\
 1 & 0 & 1 & 0 & 0 & 0 & 0 & 0 & 0 & 0 & 1 & 0 & 1 & 0 & 0 \\
 1 & 1 & 0 & 1 & 0 & 0 & 0 & 0 & 0 & 0 & 1 & 1 & 0 & 1 & 0 \\
 0 & 1 & 0 & 0 & 1 & 0 & 0 & 0 & 0 & 0 & 0 & 1 & 0 & 0 & 1 \\
\end{array}
\right] $$
\end{tiny}
The tensor product of $L(H)$ with $L(G)$ is conjugated to it 
$$ \left[
\begin{array}{ccccccccccccccc}
 1 & 1 & 1 & 1 & 1 & 1 & 1 & 1 & 1 & 1 & 1 & 1 & 0 & 0 & 0 \\
 1 & 1 & 0 & 1 & 1 & 0 & 1 & 1 & 0 & 1 & 1 & 0 & 0 & 0 & 0 \\
 1 & 0 & 1 & 1 & 0 & 1 & 1 & 0 & 1 & 1 & 0 & 1 & 0 & 0 & 0 \\
 1 & 1 & 1 & 1 & 1 & 1 & 0 & 0 & 0 & 1 & 1 & 1 & 1 & 1 & 1 \\
 1 & 1 & 0 & 1 & 1 & 0 & 0 & 0 & 0 & 1 & 1 & 0 & 1 & 1 & 0 \\
 1 & 0 & 1 & 1 & 0 & 1 & 0 & 0 & 0 & 1 & 0 & 1 & 1 & 0 & 1 \\
 1 & 1 & 1 & 0 & 0 & 0 & 1 & 1 & 1 & 0 & 0 & 0 & 0 & 0 & 0 \\
 1 & 1 & 0 & 0 & 0 & 0 & 1 & 1 & 0 & 0 & 0 & 0 & 0 & 0 & 0 \\
 1 & 0 & 1 & 0 & 0 & 0 & 1 & 0 & 1 & 0 & 0 & 0 & 0 & 0 & 0 \\
 1 & 1 & 1 & 1 & 1 & 1 & 0 & 0 & 0 & 1 & 1 & 1 & 0 & 0 & 0 \\
 1 & 1 & 0 & 1 & 1 & 0 & 0 & 0 & 0 & 1 & 1 & 0 & 0 & 0 & 0 \\
 1 & 0 & 1 & 1 & 0 & 1 & 0 & 0 & 0 & 1 & 0 & 1 & 0 & 0 & 0 \\
 0 & 0 & 0 & 1 & 1 & 1 & 0 & 0 & 0 & 0 & 0 & 0 & 1 & 1 & 1 \\
 0 & 0 & 0 & 1 & 1 & 0 & 0 & 0 & 0 & 0 & 0 & 0 & 1 & 1 & 0 \\
 0 & 0 & 0 & 1 & 0 & 1 & 0 & 0 & 0 & 0 & 0 & 0 & 1 & 0 & 1 \\
\end{array}
\right]  \; .  $$
The eigenvalues of the tensor product are the products 
$\lambda_j \mu_k$, where $\lambda_j \in \sigma(L(G))$ and
$\mu_j  \in \sigma(L(H))$. 

\begin{comment}
Import["~/text3/graphgeometry/all.m"];
t1 = LinearGraph[2]; t2 = LinearGraph[3];
L12= ConnectionLaplacian[t1, t2] 
\end{comment}

\paragraph{}
{\bf Example 4)} Here is an example of a graphic matroid. 
If $(V,E)=(\{a,b,c,d,e\},\{(14),(12)$, $(13),(23)$,$(34)\})$
is the diamond graph, then 
$\{(a,b,c)$, $(a,b,d)$,$(a,b,e)$, $(a,c,d)$,$(a,d,e)$, $(b,c,e)$, $(b,d,e)$, $(c,d,e)\}$ 
is a list of spanning trees and $G=\{(a),(b),(c)$, $(d),(e),(a,b)$, $(a,c)$, 
$(a,d)$, $(a,e)$, $(b,c)$, $(b,d)$, $(b,e)$,
$(c,d)$, $(c,e)$, $(d,e)$, $(a,b,c)$, 
$(a,b,d)$, $(a,b,e)$, $(a,c,d)$, $(a,d,e)$, $(b,c,e)$, 
$(b,d,e)$, $(c,d,e)\}$ is the graphic matroid of the diamond graph. 
It has the $f$-vector $(5,10,8)$ and Euler characteristic $G$. 
The $(23 \times 23)$-matrix $L$ has the inverse whose entries add up to $3$.

\paragraph{}
{\bf Example 5)} If $G=\{ \{1,2\},\{1\},\{2\} \}$ 
is the set of non-empty subsets of a two 
point set, then $n=3,b=2,f=1$ and $\chi(G)=2-1=1$. The matrices $L$ and $g$ are
$$ L=\left[ \begin{array}{ccc} 1 & 1 & 1 \\ 1 & 1 & 0 \\ 1 & 0 & 1 \\ \end{array} \right],
   g=\left[ \begin{array}{ccc} -1 & 1 & 1 \\ 1 & 0 & -1 \\ 1 & -1 & 0 \\ \end{array} \right] \; . $$
We see that $n=3,b=2,f=1$ and $\chi(G)=1, \phi(G)=-1$. We check $\sum_{x,y} g(x,y)=1$.
The is also $\chi(G)=b-f$.  The Fermi characteristic $\phi(G) = 1^b (-1)^f=(-1)^f =-1$
agrees with ${\rm det}(L)=-1$. 

\paragraph{}
{\bf Example 6)} If $G$ is the diamond graph complex 
generated by the set $A=\{ (1,2,3),(2,3,4) \}$, then
the connection matrix and its inverse are (using $a=-1$ for 
typographical reasons):
\begin{tiny}
$$ L=\left[\begin{array}{ccccccccccc}
1&0&0&0&1&1&0&0&0&1&0\\
0&1&0&0&1&0&1&1&0&1&1\\
0&0&1&0&0&1&1&0&1&1&1\\
0&0&0&1&0&0&0&1&1&0&1\\
1&1&0&0&1&1&1&1&0&1&1\\
1&0&1&0&1&1&1&0&1&1&1\\
0&1&1&0&1&1&1&1&1&1&1\\
0&1&0&1&1&0&1&1&1&1&1\\
0&0&1&1&0&1&1&1&1&1&1\\
1&1&1&0&1&1&1&1&1&1&1\\
0&1&1&1&1&1&1&1&1&1&1\\ \end{array}\right] g=\left[
                  \begin{array}{ccccccccccc}
0&0&0&0&0&0&a&0&0&1&0\\
0&0&1&0&0&a&a&0&a&1&1\\
0&1&0&0&a&0&a&a&0&1&1\\
0&0&0&0&0&0&a&0&0&0&1\\
0&0&a&0&0&1&1&0&0&a&0\\
0&a&0&0&1&0&1&0&0&a&0\\
a&a&a&a&1&1&1&1&1&a&a\\
0&0&a&0&0&0&1&0&1&0&a\\
0&a&0&0&0&0&1&1&0&0&a\\
1&1&1&0&a&a&a&0&0&1&0\\
0&1&1&1&0&0&a&a&a&0&1\\ \end{array} \right] \; $$
\end{tiny}
The entries of $g$ add up to $1$ which is the Euler characteristic $\chi(G)$ of $G$. 

\paragraph{}
{\bf Example 7)} The set $A=\{ (1,2),(2,3),(3,1) \}$ generates a $1$-dimensional complex 
which is not the Whitney complex of a graph. It is the smallest $1$-dimensional sphere in
the category of simplicial complexes. We have $n=6,b=3,f=3$ and $\chi(G)=3-3=0$. 
The matrices are \\
$$ L=\left[ \begin{array}{cccccc}
                   1 & 0 & 0 & 1 & 0 & 1 \\
                   0 & 1 & 0 & 1 & 1 & 0 \\
                   0 & 0 & 1 & 0 & 1 & 1 \\
                   1 & 1 & 0 & 1 & 1 & 1 \\
                   0 & 1 & 1 & 1 & 1 & 1 \\
                   1 & 0 & 1 & 1 & 1 & 1 \\
                  \end{array} \right], 
 g=\left[ \begin{array}{cccccc}
                   -1 & -1 & -1 & 1 & 0 & 1 \\
                   -1 & -1 & -1 & 1 & 1 & 0 \\
                   -1 & -1 & -1 & 0 & 1 & 1 \\
                   1 & 1 & 0 & -1 & 0 & 0 \\
                   0 & 1 & 1 & 0 & -1 & 0 \\
                   1 & 0 & 1 & 0 & 0 & -1 \\
                  \end{array} \right] \; . $$

\section{Poincar\'e-Hopf} 

\paragraph{}
Every simplicial complex $G$ defines a finite simple graph $G_1=(V,E)$, where $V$ consists
of the sets of $G$ and where two vertices in $V$ are connected by an edge if one is
contained in the other as an element in $G$. The
Whitney complex of this graph is called the {\bf Barycentric refinement} $G_1$
of $G$. One can define the Barycentric refinement also as the subset $G_1$ of the power set $2^G$,
where each simplex $A \in G_1$ has the property that if $x,y \in A$
then either $x \subset y$ or $y \subset x$, where the subset relation refers to the elements $x,y$ 
of $G_1$ as subsets of $G$.
 
\paragraph{}
A function $f:G \to \RR$ defines a function on the vertex set of the Barycentric 
refinement $G_1$. Let $S(x)$ denote the {\bf unit sphere} of $x$: 
it consists of all simplices $y$ in $G$ which are either strictly 
contained in $x$ or all simplices $y$ which strictly contain $x$. The set of sets
$S(x)$ is not a simplicial complex, but it defines a subgraph of $G_1$, which carries a
Whitney complex which is a simplicial complex. The following Poincar\'e-Hopf theorem applies
for general graphs $(V,E)$ and not only for graphs $G_1$ defined by simplicial complexes. 

\paragraph{}
A function $f:G \to \mathbb{R}$ is {\bf locally injective} if $f(x) \neq f(y)$ for any in 
$G_1$ connected pair $x,y$ in $G_1$. A locally injective function is also called a 
{\bf coloring} of $G_1$ meaning that every set $\{ f = c \}$ is independent.
Define the {\bf index} $i_f(x) = 1-\chi(S^-_f(x))$, where $S^-_f(x)$ 
is the sub-complex of all $y \in S(x) \subset G_1$ with $y \subset x$. 

\begin{lemma}[Poincar\'e-Hopf, \cite{poincarehopf} 2012]  
$\chi(G) = \sum_x i_f(x)$ for a locally injective $f$.
\end{lemma}
\begin{proof}
The formula holds in general for any finite simple graph $G=(V,E)$ with 
locally injective function $f$ defined on $V$. To prove it, use induction with respect to the
number $n$ of vertices. Take a vertex $v \in G$ for which $f$ is locally maximal. 
Let $H$ be the subgraph with $v$ and connections taken away. Since $\chi$
is a valuation, we have $\chi(G) = \chi(H) + \chi(B(x)) - \chi(S(x))$. 
Because $\chi(B(x))=1$ and $\chi(S(x)) = \chi(S_f^-(x))$ and by the
induction assumption $\chi(H) = \sum_y i_f(y)$ we have 
$\chi(G) = \sum_{y \neq x} i_f(y) + i_f(x)$. 
We have used that that in a graph, for any vertex $v$ the unit ball 
$B(v)$ has Euler characteristic $1$. 
\end{proof}

\paragraph{}
If $G_1$ is the Barycentric refinement of a complex $G$, one can look at
the function $f(x)={\rm dim}(x)$. It is a coloring with a minimal number of colors
as the chromatic number is equal to the clique number. 
The set $S^-(x)$ is now a sphere complex as it is the boundary of the simplex $x$.
Its Euler characteristic is either $0$ or $2$. 
The set of sets $S^+(x)=S(x) \setminus S^-(x)$ is the {\bf star} of $x$. It is not 
a simplicial complex in general. 

\paragraph{}
In the case $f(x)={\rm dim}(x)$, we have $i_f(x) =\omega(x)$ as it is the genus of
the sphere complex $S^-(x)$. Poincar\'e-Hopf
now tells that $\chi(G_1) = \chi(G)$. The fact that Euler characteristic is a combinatorial
invariant follows also from the explicit relation between the $f$-vectors of $G$ and
$G_1$. The $f$-vector gets multiplied with the Barycentric refinement operator $A=i! S(j,i)$,
where $S(j,i)$ are the {\bf Stirling numbers of the second kind}. As the transpose $A^T$ has a 
unique eigenvector $(1,-1,1,-1,\dots)$ with eigenvalue $1$, the Euler characteristic is the 
only valuation (linear functional on $f$-vectors) with the property that it is invariant under 
Barycentric refinements. Poincar\'e-Hopf confirms in particular that $\chi$ is invariant 
under refinement.

\paragraph{}
For a generalization in which the Euler characteristic is replaced by the $f$-function 
$f_G(t) = 1+f_0 t + f_1 t^2 + \cdots f_d t^{d+1}$, see \cite{parametrizedpoincarehopf}. 
Given a locally injective function $g$ on the vertex set of a graph, then 
one has $f_G(t) = 1+t \sum_{x \in V} f_{S_g(x)}(t)$. 

\section{Gauss-Bonnet}

\paragraph{}
If $\Omega$ is the set of locally injective functions on $G$ and $\A$ a $\sigma$-algebra
on $\Omega$ and $P$ a probability measure on the measure space $(\Omega,\A)$ we have a probability 
space $(\Omega,\A,P)$. The {\bf curvature} defined by this probability space is 
defined as the expectation $K(x)={\rm E}[i_f(x)]$. It obviously satisfies the Gauss-Bonnet 
formula as it is an average of Poincar\'e-Hopf theorems:

\begin{coro}[Gauss-Bonnet]
$\chi(G) = \sum_{x \in G} K(x)$. 
\end{coro}
\begin{proof}
As $\chi(G) = \sum_x i_f(x)$, we have $\chi(f)=\sum_x {\rm E}[i_f(x)]$. 
\end{proof}

\paragraph{}
This is a cheap approach to Gauss-Bonnet and also holds for compact Riemannian manifolds $M$. 
There are probability spaces which produces Euler curvature used in the Gauss-Bonnet-Chern theorem. 
One can for example Nash embed $M$ isometrically in an ambient Euclidean space $E$, 
then take the probability space of all linear functions in $E$. They produce a probability space of
functions on $M$. An possibly different curvature is obtained by taking the normalized 
volume measure on $M$ as $P$ using the heat kernel functions 
$\Omega = \{ f_y(x) = [e^{-\tau L_0}]_{xy} \}$ where $L_0$ is the Laplacian on scalars. 

\paragraph{}
If $\Omega = \prod_{x \in G} [-1,1]$ and $P=\prod_{x \in G} dx/2$
of if $P$ is the uniform counting measure on all $c$-colorings, where $c$ is
the chromatic number, we get the {\bf Levitt curvature} 
$$  K(x) = \sum_{y, x \subset y} \frac{\omega(y)}{|y|} 
         = 1-\frac{V_0(x)}{2}+\frac{V_1(x)}{3}-\frac{V_2(x)}{4} \dots \; ,  $$ 
where $V_k(x)$ is the number of $k$-dimensional simplices in $S(x)$. 
In the $2$-dimensional case, the curvature is $K(x)=1-V_0/2+V_1/3$ which is 
$K(x)=1-{\rm deg}(x)/6$ in the case of a triangulation, a case known for a long time.
The general case appeared in \cite{Levitt1992}. 
It was rediscovered in a geometric setting 
in \cite{cherngaussbonnet} where it was placed as a Gauss-Bonnet result. 
The integral theoretic picture of seeing curvature as an expectation of indices is 
adapted from integral geometry in the continuum \cite{indexexpectation,colorcurvature}.

\paragraph{}
As for Poincar\'e-Hopf, also Gauss-Bonnet does not require the graph to be 
the Barycentric refinement of a complex. It holds for any finite simple graph. 
Both the index function $i_f(x)$ as well as the curvature function $K(x)$ are defined
on the vertices but unlike curvature, the index is always integer-valued.
Poincar\'e-Hopf is even more general: let $\phi$ be a map from the set $V_1$ of 
vertices in the Barycentric refinement $G_1$ to the set $V$ of vertices in $G$. 
Now take the values $\omega(x)$ on $V_1$ which add up to Euler characteristic and
place them to the vertex $\phi(x)$. If $\phi(x) = \{ v \in x \; | \; f(v) <f(w), \forall w \in x \}$
then we get Poincar\'e-Hopf. Taking a probability space of functions corresponds then to 
a probability space of transition maps, a Markov process. If we distribute the value $\omega(x)$
equally to its zero-dimensional part, the Levitt version of Gauss-Bonnet results. 

\section{Abstract finite CW complexes}

\paragraph{}
A {\bf finite abstract CW complex} is a geometric, combinatorial object which 
like $\Delta$-sets or simplicial sets generalizes simplicial complexes. 
The definition does not tap into Euclidean structures like classical definition 
of $CW$-complexes. But it closely follows the definition used in the continuum. 
The constructive definition uses a gradual build-up which allows to make proofs easier. 
Any finite abstract simplicial complex is a CW complex but CW complexes 
are more general. 

\paragraph{}
The empty complex $0=\{\}$ is declared to be a CW-complex and to be a $(-1)$-sphere. 
This complex does not contain any cell. To make an extension of a given a complex $H$,
choose a sphere $S$ within $H$ and add a {\bf cell} $x$ which has $S$ as the boundary. 
Given a complex $H$ with $n$ cells and a sub-complex $K$ which is a $d$-sphere, we
add a new $(d+1)$-cell given as the join $x=H+1$.
The integer $d+1$ attached to the cell is the {\bf dimension} of the cell.
The cells in $S(x)$ are declared to be {\bf a subset} of $x$ or {\bf contained in} $x$.
The maximal dimension of a cell $x$ is the {\bf dimension} of the CW complex. 

\paragraph{}
The {\bf unit sphere} $S(x)$ of a cell $x$ is the union of all cells either contained in $x$
or all cells which contain $x$. 
A CW-complex $G$ is called {\bf contractible}, \footnote{We identify here contractible and 
collapsible and would use the terminology ``homotopic to $1$" for the wider equivalence relation.}
if there exists a cell $x$ which has a contractible unit sphere $S(x)$ 
so that $G \setminus x$ is contractible. To start of the inductive definition 
assume that the empty complex $0 = \emptyset$ is a $(-1)$-dimensional sphere declared to be
not contractible and that the $1$-point complex $1 = K_1$ is declared to be contractible. 
A CW-complex $G$ of maximal dimension $d$ is a {\bf $d$-sphere} if it is a $d$-complex
for which removing any single cell produces a contractible complex. A CW-complex $G$ is a 
{\bf $d$-complex}, if all unit spheres $S(x)$ are $(d-1)$-spheres. We do not require a 
CW complex to be a $d$-complex. 

\paragraph{}
A connection Laplacian $L$ for a given CW complex is defined in the same way if we define
two cells to intersect if they have a common sub-cell. 
Assume the connection Laplacian $L$ for $G$ has been constructed, the connection Laplacian $L$ 
of the enlarged complex $G+x$ is a $(n+1) \times (n+1)$ matrix, where $L(x,y)=L(y,x) = 1$ if 
$y \cap (S(x) \cup \{x\}) \neq \emptyset$ and $0$ else. In the case of a 
simplicial complex, which corresponds to the already given definition where 
$L(x,y)=1$ if $x$ and $y$ intersect and $L(x,y)=0$ else. 

\paragraph{}
The {\bf Barycentric refinement} of a CW complex $G$ is the Whitney complex of a graph $G_1$.
The cells of $G$ are the vertices of $G_1$ and two cells are connected, 
if one is contained in the other. 
The {\bf connection graph} $G'$ of $G$ is the graph with the same vertices than $G_1$, but
where two cells are connected if they intersect. We can write $L=1+A$, where $A$ is the 
adjacency matrix of $G'$ so that the determinant of $L$ is the
{\bf Fredholm determinant} of $A$.

\paragraph{}
When adding a cell to a complex, the Euler characteristic changes by 
$\chi(H) \to \chi(G)=\chi(H)+(1-\chi(S(x)))$. 
We show in the over next section that the Fermi characteristic changes by 
$\psi(H) \to \psi(G) = \psi(H) (1-\chi(S(x)))$.

\section{Valuations}

\paragraph{}
A {\bf valuation} $X$ on a simplicial complex $G$ is a real-valued function on the set of
sub-complexes of $G$ satisfying the {\bf valuation property} $X(A \cup B) + X(A \cap B)=X(A)+X(B)$
for all sub-complexes $A,B$. A complex is called {\bf complete} if it is of the form 
$G=2^A \setminus \{ \emptyset \}$ for some finite set $A$. Every simplex $x \in G$ naturally 
defines a {\bf complete complex} $W^-(x) = \{ y \neq \emptyset \; | \; y \subset x \}$.
There is an obvious bijection between
complete sub=complexes of $G$ and subsets of $G$. There is a difference however. Most
notably, the empty complex $0=\emptyset$ is a complex but not a set=simplex in a simplicial complex.
Some authors add the empty set $\emptyset$ to any simplicial complex which leads to the {\bf reduced Euler
characteristic} $\chi(G)-1$, where the {\bf void} $\emptyset$ counts as a $(-1)$-dimensional simplex.
We prefer to avoid voids and see the empty-set as a $(-1)$-dimensional sphere of Euler characteristic 
$0$ and not a simplex and $1-\chi(G)$ as a {\bf genus}. 

\begin{lemma}
A valuation $X$ which satisfies $X(A)=1$ for every complete sub-complex $A$ of $G$
must be equal to the Euler characteristic of $G$.
\label{valuationlemma}
\end{lemma}
\begin{proof}
Since the Euler characteristic of a complete complex is $1$ we only have to
show uniqueness. By the discrete Hadwiger theorem \cite{KlainRota}, any valuation
is of the form
$X(A) = X \cdot f(A)$, where $f(A)$ is the f-vector of $A$ and $X$ is a vector.
If $G$ has maximal dimension $d$, then the space of valuations has dimension $d+1$.
If $X(A)=1$ for every complete graph, this means $X \cdot \vec{f}(K_k) = 1$, for the
$f$-vectors $\vec{f}(K_k)$ of $K_k$. But since these vectors form a basis in the 
vector space $\RR^{d+1}$, we also have uniqueness and $X=\chi$. 
\end{proof}

\paragraph{}
We have defined a valuation only for simplicial complexes so far.
For $CW$ complexes, we can define a valuation to be a linear functional $X(G) = \sum_i X_i f_i(G)$
on the $f$-vector $(f_0(G),f_1(G), \dots, f_d(G))$ of the complex, where $f_k(G)$
counts the number of $k$-dimensional cells in $G$. The same definition applies
for the Cartesian products of two simplicial complexes or for signed complexes by declaring
$X(-G)=-X(G)$ after extending the ``disjoint union monoid" of simplicial complexes 
to a group.

\section{Paths}

\paragraph{}
If $A$ is the adjacency matrix of a graph $\Gamma$, the determinant  ${\rm det}(A)$ is a partition 
function or a ``path integral", in which the underlying paths are 
fixed-point-free signed permutations 
of the vertices. The determinant generates {\bf derangements} $\pi$ for which
$x \to \pi(x) \to \pi(\pi(x)) \dots$ defines oriented paths.
The Fredholm determinant $\zeta(\Gamma) = {\rm det}(1+A)$ is a partition function for all 
oriented cyclic paths $\pi$ in the graph as $x \to \pi(x)$ can now also have pairs $(a,b) \in E$ as 
transpositions and vertices $v \in V$ as fixed points. In short,
$$ {\rm det}(1+A) = \sum_{\gamma} (-1)^{|\gamma|} $$
summing over all one-dimensional oriented cyclic paths $\gamma$ of $\Gamma$
and where $(-1)^{|\gamma|} = {\rm sign}(\gamma)=\phi(\gamma)$, the signature of the corresponding permutation,
is also the Fermi number involving one-dimensional parts of $\gamma$ if $\gamma$ seen as a 
$1$-dimensional complex. 

\paragraph{}
The unimodularity theorem tells that if the graph $\Gamma$ is the connection graph $G'$ 
of a simplicial complex $G$ and $L$ is the connection Laplacian of $G$, then 
$$ {\rm det}(L) = \sum_{\gamma \subset G} (-1)^{|\gamma|}  = \prod_{x \in G} (-1)^{|x|-1} = \phi(G) \; . $$
The energy theorem tells that the Euler characteristic is 
$$ \sum_{x,y} L^{-1}_{x,y} = \sum_{x \in G} (-1)^{|x|-1}  = \chi(G)   \; . $$
Cramer's determinant formula shows that the left hand side has a path interpretation too. In other words,
both the determinant of $L$ as well as the Euler characteristic of $G$ have a path integral representation
summing over one-dimensional closed oriented loops of the complex.

\paragraph{}
If ${\rm per}(A)$ is the {\bf permanent} of $A$, then 
${\rm per}(A)$ is the number {\bf derangements} of the vertex set of
the graph while ${\rm per}(1+A)$ is the number of all
{\bf permutations} of the vertex set.. For $G=K_n$ for example, 
${\rm per}(1+A(K_n))$ generates the {\bf permutation sequence} $1,2,6,24,120,720, \dots$
while the permanent of the adjacency matrix ${\rm per}(A(K_n))$ generates
the {\bf derangement sequence} $0,1,2,9,44,265, \dots$. 

\paragraph{}
Every path has comes with a sign, the signature of its permutation. 
The unimodularity theorem assures that the number of even paths and the number of odd paths 
in a connection graph differ by $1$. The following lemma will be used later. It is a special case of
the muliplicative Poincar\'e-Hopf lemma. If $G$ is a $CW$-complex and $G_x = G+_H x$ is the extended
complex, then, we can look at the new connection graph $G_x' = G' +_{H'} x$. 

\begin{lemma}[Fredholm extension lemma]
If $H$ is a complete subcomplex of $G$ then $\psi( G' +_{H'} x ) = 0$. 
\label{fredholmextension}
\end{lemma}
\begin{proof}
If $H$ is a complete subcomplex of $G$, then it defines a vertex $h$ in $G'$. 
If a path does not hit, then it can be paired with a path adding the $hx$ path 
of length $2$. If a path does hit, then it has to to hit a neighbor $k$ which 
is also connected to $x$. We can now pair such a piece $hk$ with the extended path 
$hxk$ and again get a cancellation. How look at all paths which do hit $h$ and
do not hit $hk$ etc. We see that the paths can be partitioned into a finite set
of paths where each can be paired with an extended path with opposite sign. 
\end{proof} 

\paragraph{}
Note that $H$ is a subcomplex of $G$, not of $G'$. We build then the complex 
$G_x=G +_H x$ which produces the connection graph $G_x'=G' +_{H'} x$. 
The Fredholm extension lemma tells that the Euler characteristic of $H$
and not of $H'$ matters. This is essential as $\chi(H)$ and $\chi(H')$ differ in general. 
Comparing the Fredholm determinant of $G'$ and $(G +_H x)'$ is not the right thing. 
The added vertex $x$ over $H$ attaches a new cell to the CW complex. Also 
on the level of graphs, the extension lemma would fail as the join over $H$
adds a lot more simplices on the connection level. The Fredholm determinant of 
the expanded cell complex $G_x$ which on the connection graph level makes a 
cone extension over $H'$.

\section{Proof of unimodularity} 

\paragraph{}
The key of the proof is a {\bf multiplicative Poincar\'e-Hopf result} for CW complexes which immediately 
proves the unimodularity theorem. If a newly added cell $x$ is odd-dimensional, then $S(x)$ is 
an even-dimensional sphere with Euler characteristic $2$ and $(1-\chi(S(x)))=-1$ switches the sign. If 
the newly added cell $x$ is even-dimensional, then $S(x)$
is odd dimensional. In this case $1-\chi(S(x))=1$ and the sign of the 
product stays the same. The following theorem, it is not assumed that the glue $A=S(x)$ 
is a sphere. It applies therefore for more general complexes.

\begin{propo}[Multiplicative Poincar\'e Hopf]
If $x$ is a new cell attached to a sub-complex $A=S(x)$ of $G$,
then $\psi(G \cup_A \{x\}) = \psi(G) (1-\chi(A))$.
\end{propo}
\begin{proof}
(i) The map 
$$   Y: A \to (\psi(G \cup_A x) - \psi(G)) $$ 
is a valuation. \\
(ii) The valuation satisfies $\psi(G \cup_A x)=0$ if $A$ is a
complete subgraph. This follows from the extension lemma 
\ref{fredholmextension}. \\
(iii) It follows from (ii) that $Y(A)=-\psi(G)$ if $A$ is a complete subgraph.  \\
(iv) By Lemma~\ref{valuationlemma}, $Y$ must be the Euler characteristic
$Y(A)=-\chi(A)$.  But this is equivalent to the claim 
$\psi(G \cup_A \{x\}) = \psi(G) (1-\chi(A))$.
\end{proof}

\paragraph{}
An algebraic argument would use the Laplace expansion of the 
connection matrix $of G \cup_A \{ x\}$ with respect to the newly added
column. We will see that in the proof of Theorem~(\ref{2}). However, the
analysis of the minors requires a similar insight into the path expansion 
as in Lemma~\ref{fredholmextension}. 

\paragraph{}
If more general objects $H$ than spheres would be used in the 
construction $G \to G \cup_H \{x\}$,
the interface $H$ is a general sub-complex of $G$, then the unimodularity theorem would 
fail in general as the Fermi characteristic
can become zero. For a CW complex the subcomplexes $H$ consists of spheres.
This assures that $1-\chi(S(x))$ is always $1$ or $-1$.

\paragraph{}
Every simplicial complex is a CW complex: first start with $0$-dimensional cells, the
vertices, then add one-dimensional simplices $x$. Every
sphere $S(x)$ of such a simplex $x$ is a $0$-dimensional sphere. After having added all $1$-dimensional simplex, each
triangle is still a one-dimensional complex $C_3$, the $1$-skeleton of $K_3$ and a sphere. Now add
the two-dimensional simplices. This converts a one-dimensional cyclic cell complex $C_3$ into a
two dimensional complex $K_3$. We can build up any simplicial complex recursively by starting with
sets of cardinality $1$, then adding sets of cardinality $2$, then $3$ etc. It then follows from
the Euler characteristic or Fermi characteristic changes that
$\chi(G) = \sum_x \omega(x)$ and 
$\psi(G) = \prod_x \omega(x)$, where $\omega(x) = (-1)^{{\rm dim}(x)} = 1-\chi(S^-(x))$
and $S^-(x)$ is the unit sphere of the cell $x$ at the moment it has been added to the CW-complex. 

\paragraph{}
Discrete CW complexes are strictly more general than simplicial complexes: for example
add a cell $x$ to $C_4$. Its Barycentric refinement is
the Whitney complex of a wheel graph. It has only $9=8+1$ cells, where the $8$ cells came from the 
circular graph $C_4$ generated by $\{ (1,2),(2,3),(3,4),(4,1) \}$. The wheel graph $W_4$
obtained as $W_4 = C_4 + 1$ the cone extension on the other hand has more cells:
as a Whitney complex, the wheel graph $W_4$ has $5$ vertices, $8$ edges and $4$ triangles 
leading to a finite simplicial complex with $17$ cells. Topologically the two CW complexes are equivalent,
but the first has $9$ cells, the second one has $17$ cells. As in the continuum, discrete CW complexes have
both practical advantages as well as proof theoretical advantages. 

\paragraph{}
Here is an other example. The two-dimensional cube $K_2 \times K_2$ is not the Whitney complex of a graph. 
It can also not be written as a simplicial complex. The Cartesian product of two simplicial complexes is 
not a simplicial complex if both factors are positive dimensional. We can see $K_2 \times K_2$ however as a
CW-complex with $3 \cdot 3=9$ cells. The Barycentric refinement of this element is the Whitney complex of
a graph $G_1=(V,E)$, where $V$ is the set of $9$ cells and where $(a,b) \in E$, if either $a \subset b$ or
$b \subset a$. Completely analog is the boundary of the solid cube $K_2 \times K_2 \times K_2$ which we call
the "cube" or hexaedron. It can be realized by starting with the one-dimensional graph representing the cube
and adding 6 cells. This gives a CW complex with six $2$-dimensional cells, twelf
$1$-dimensional faces and eight $0$-dimensional cells. 
It is only as a discrete CW complex that we regain the familiar picture of the cube as a
$2$-dimensional sphere.

\section{Joins} 

\paragraph{}
Let us first define the join for simplicial complexes.
Given two sets $x,y \in G$, let $x+y$ be the disjoint union and $x \cup y$ the union. 
Given two complexes $G$ and $H$, then 
$G+H = G \cup H \cup \{ x \cup y \; | \; x \in G, y \in H \}$ is a complex 
called the {\bf Zykov join}. It is the discrete analogue of the join in topology.

\paragraph{}
The complex $G+1$ is a {\bf cone extension} of $G$. If $P_2=\{\{1\},\{2\} \}$ 
denotes the zero-dimensional sphere, the complex $G+P_2$ is the {\bf suspension}. 
The Zykov monoid has the class of spheres as a submonoid. The complex $P_2 + P_2$ is the circle $C_4$,
the complex $P_2+C_4=O$ is the octahedron. The complex $n P_2 = P_2 + P_2 +  \dots + P_2$
is a $(n-1)$-dimensional sphere. A complex is an {\bf additive Zykov prime}, 
if it can not be written as $H+K$, where $H,K$ are complexes. 

\paragraph{}
To define the join for CW complexes $G_0 \to_{S_0} G_1 \to_{S_1} G_2 \to \cdots \to_{S_{n-1}} G_{n}$
and $H_0 \to_{T_0} H_1 \to_{T_1} H_2 \to \cdots \to_{T_{n-1}} T_{n}$ using spheres $S_i,T_i$. 
Now declare $T_k + S_k$ to be the spheres of $G+H$, then build up $G+H$ by building $G_k + H_l$.

\paragraph{}
The following lemma holds for general CW complexes:

\begin{lemma}
If $G$ is contractible and $H$ is arbitrary, then $G+H$ is contractible. 
\end{lemma}
\begin{proof}
Use induction with respect to the number of cells. Assume $G=K+_A x$,
where $A$ is contractible, then $G+H = K+H +_{A + K} x$. By induction,
$A+K$ is contractible and $K+H$ is contractible. Therefore $G+H$ is contractible.
\end{proof}

\paragraph{}
For example, the cone extension $G+1$ of any complex is contractible. The just
observed statement implies:

\begin{coro}
Contractible complexes form a submonoid of all simplicial complexes. An additive prime
factor of a non-collapsible complex is not-collapsible.
\end{coro}

\paragraph{}
For example, $C_4 = S_0 + S_0$, where $S_0$ is the zero dimensional $2$-point complex
which is the $0$-dimensional sphere. $C_4$ is not contractible so that also $S_0$ is not
contractible. 

\paragraph{}
A {\bf d-ball} is a punctured sphere, a $d$-sphere for which one vertex has been taken away.
By definition, a $d$-ball is contractible. 
By definition a $d$-ball is a $d$-complex with boundary which has a $(d-1)$-sphere as a boundary. 
For example, the cone extension $G+x$ of a sphere $G$ is a ball because making an other
cone extension $G+x+y$ is a suspension $G+S_0$ which is a sphere so that $G+x=G+S_0-y$ is a ball. 
Unlike spheres, balls do not form a sub-monoid of the join complex. 
The ball $1=K_1$ for example gives gives $1+1=K_2$ which is not a 1-simplex and not a ball. 

\paragraph{}

\begin{lemma}
If $G$ and $H$ are both spheres, then $G+H$ is a sphere.
If $G$ is a sphere and $H$ is a ball, then $G+H$ is a ball. 
\end{lemma}
\begin{proof}
The proof is inductive. For $G=H=0$, we have $G+H=0$ and
for $G=0,H$ we have $G+H=H$. \\
Write $G=K +_A x$, where $A$ is a smaller dimensional sphere and $K$ is a 
smaller dimensional ball.  Now $G+H = K+H +_{A+H} x$. 
We know $K+H$ is a ball by induction assumption 
and that $A+H$ is a sphere by induction assumption.
\end{proof}

\section{Hyperbolicity} 

\paragraph{}
Given a simplicial complex $A$, its {\bf genus} is defined as 
$\gamma(A) = 1-\chi(A)$.
The reason for the name is that if $A$ is one-dimensional connected complex then $\chi(A)=b_0-b_1=1-b_1$
is the Euler-Poincar\'e formula relating the combinatorial and cohomological 
Euler characteristic then $b_1=\gamma(A)$ is the genus of the curve, the number 
of ``holes". The analogy is less established in higher dimensions:
for surfaces already, where 
$\chi(A)=b_0-b_1+b_2=2-b_1$ it is custom to define the genus as 
$1-\chi(A)/2$. That the notion is in general natural can be seen from the following product formula

\begin{lemma}[Product formula]
$\gamma(A + B) = \gamma(A) \gamma(B)$.
\end{lemma}
\begin{proof}
If $f(t)=1+\sum_{k=1}^d f_{k} t^{k+1} =            
1+f_0 t + f_1 t^2 + \dots $ is the 
{\bf generating function} for the $f$-vector $(f_0,f_1, \dots)$ of $G$,
then $\chi(G)=-f(-1)$. 
It satisfies therefore $f_{A+B} = f_A f_B$. This implies
$\chi(A+B) = \chi(A) + \chi(B) - \chi(A) \chi(B)$ because 
$\chi(G)=1-f_G(-1)$.
The genus $\gamma(G)=1-\chi(G)$ is multiplicative.
\end{proof}

\paragraph{}
It implies immediately for the complete graph
$K_n=n=1+1+\dots +1$ that $\gamma(n)=1$ and that a {\bf suspension}
$G \to G+S_0$ changes the sign of the genus. 

\paragraph{}
Let $f$ be a locally injective function and $S(x)$ the unit sphere of $x$ in 
the graph $G_1$. It is a simplicial complex in which the vertices are
the simplices which are either contained in $x$ or which are simplices containing $x$. 
Let $S_f^-(x)=\{ y \in S(x) | f(y)<f(x) \}$ and $S_f^+(x) = \{ y \in S(x) | f(y)>f(x)\}$.
All graphs $S(x),S_f^-(x)$ and $S_f^+(x)$ stand for their Whitney complexes.

\begin{lemma}[Hyperbolic structure]
For the function $f={\rm dim}$ on a simplicial complex $G$, 
we have $S(x) = S_f^-(x) + S_f^+(x)$
\end{lemma}
\begin{proof}
This is stated as Lemma~(1) in \cite{Spheregeometry}.
Every simplex $y$ in $S(x)$ is either a simplex in $S^-_f(x)$ or a simplex
in $S^+_f(x)$.  The function $f$ satisfies $f(y^-)<f(y)$ if $y^- \subset y$ 
and $f(y^+)>f(y)$ if $y^+ \supset y$.
\end{proof}

\paragraph{}
For a general function $f$ we don't have $S(x) = S_f^-(x) + S_f^+(x)$ as the
elements in the stable part $S_f^-(x)$ are not necessarily connected to the
unstable part. They might not even be connected at all in $S(x)$. 
% It follows that for x of neither 0 nor maximal dimension, $S(x)$ is not prime.

\paragraph{}
{\bf Example.} If $G=W_4$ is the wheel graph with $4$ spikes 
and $f(x)$ is alternatively $1$ or $-1$ on the boundary $S_4$ 
and $0$ on the central point $0$, then $S(0)=C_4$ and $S^+(x)$ is a $0$-dimensional sphere
and $S^+(0)$ an other zero-dimensional sphere. The sphere $S(x)$ is 
the sum of the two spheres. 

\paragraph{}
{\bf Example.} If $G= \{\{1\},\{2\},\{3\},\{4\},\{5\},\{6\}$,$\{1,3\},\{1,4\},\{1,5\}$,
$\{1,6\},\{2,5\},\{3,6\},\{4,5\}$,$\{1,3,6\},\{1,4,5\}\}$ and $x=\{1,3\}$. 
The set of vertices connected to $x$ are $\{\{1\},\{3\},\{1,3,6\}\}$.
This is not a simplicial complex but the graph with these vertices and connecting two
if one is contained in the other is what is the sphere $S(x)$ in $G_1$. 
We have $S_f^-(x) = \{ \{1\}, \{3 \} \}$ and $S_f^+(x) = \{ \{ 1,3,6 \} \}$. 
Now $1-\chi(S(x)) = 1-1=0$ and $1-\chi(S^-(x)) = 1-2=-1$ and 
$s-\chi(S^+(x)) = 1-1 =0$. 

\paragraph{}
For any $x \in G$, define the {\bf index} $i(x) = 1-\chi(S(x))$. 
For any locally injective function, the index $i_f(x)$ of $f$ and the index $i_{-f}(x)$ of $-f$
are linked by

\begin{coro}[Dual index]
$i_{f}(x) i_{-f}(x) = i(x)$. 
\end{coro}
\begin{proof}
This is a reformulation as $i_{-f}(x)=1-\chi(S^+_f(x))$
and $i_f(x) = 1-\chi(S^-_f(x))$. 
\end{proof}

\paragraph{}
It follows for $f = {\rm dim}$ that 
$$    k(x) = \omega(x) (1-\chi(S(x))) $$
is a curvature: 

\begin{coro}[Sphere Gauss-Bonnet]
$\sum_x k(x) = \chi(G)$. 
\label{spheregaussbonnet}
\end{coro}
\begin{proof}
The sum $\sum_x k(x)$ is equal to the sum $\sum_x \omega(x) (1-\chi(S(x)))$
which is $\sum_x i_f(x)$ where $f(x)=-{\rm dim(x)}$. The result follows now from 
Poincar\'e-Hopf.  
\end{proof}

\section{McKean-Singer formula}

\paragraph{}
The {\bf trace} of a matrix $L$ on the finite dimensional Hilbert space of all 
functions $G \to \RR$ is defined as ${\rm tr}(L) = \sum_x L(x,x)$. 
The {\bf super trace} is defined as 
$$  {\rm str}(L) = \sum_x \omega(x) L(x,x) \; . $$

\paragraph{}
If $g$ is the inverse of $L$, then its diagonal entries are the genus of $S(x)$. 

\begin{lemma}[Green's function formula]
$g(x,x) = (1-\chi(S(x)))$. 
\end{lemma}
\begin{proof}
By the Cramer formula, $g(x,x)$ is $\psi(G \setminus x)/\psi(G)$ but
by the multiplicative Poincar\'e-Hopf theorem, 
this is $1-\chi(S(x))$, where $S(x)$ is the unit sphere in $G_1$. 
\end{proof} 

\paragraph{}
The sum 
$$   V(x) = \sum_{y \in G} g(x,y) $$ 
is the {\bf total potential energy} of $x$. We have now:

\begin{coro}[Mc-Kean Singer]
$\str(L^{-1}) = \chi(G)$.
\end{coro}
\begin{proof}
This is a reformulation of the sphere Gauss-Bonnet formula
\ref{spheregaussbonnet} that
$$ \chi(G) = \sum_x \omega(x) (1-\chi(S(x))) = \sum_x k(x) \; , $$
which used that $k(x) = i_f(x)$ for $f=-{\rm dim}$. 
\end{proof} 

\paragraph{}
This is the analog of the McKean-Singer formula \cite{McKeanSinger}
$$  \str(e^{-t H}) = \chi(G)  \; , $$ 
which holds for the Hodge 
Laplacian $H=(d+d^*)^2$ of $G$ \cite{knillmckeansinger}. The discrete
case for the Hodge Laplacian follows closely the continuum proof given
in \cite{Cycon}. Here, for the connection Laplacian, we have 
$\str(L^k) = \chi(G)$ for $k=-1,0,1$, where the cases $k=0,1$ are
the definition of Euler characteristic. For $|k| \geq 2$, there
is no such identity any more in general. 

\begin{coro}
For any simplicial complex $\sum_x \omega(x) \chi(S(x))  = 0$. 
\end{coro} 
\begin{proof}
This follows directly from 
$\chi(G) = \sum_x \omega(x) (1-\chi(S(x)))= \sum_x \omega(x)$. 
\end{proof}

% <<all.m; s=BarycentricRaw[ErdoesRenyi[13,0.4]]; v=VertexList[s]; dim=Map[Length,v]-1; Sum[(-1)^Length[v[[k]]] EulerChi[UnitSphere[s,v[[k]]]],{k,Length[v]}]

\section{Proof of the Energy theorem}

\paragraph{}
Denote by $B(x)=B_{G'}(x)$ the {\bf unit ball} of $x$ in the connection graph $G'$.
We can write the {\bf connection vertex degree} $d(x) = d_{G'}(x)$ of a vertex $x$ in terms
of stable spheres $S^+(y) = \{ z \in S(y)\; | \; z \subset y \}$ in $G_1$. 

\begin{lemma}[Unit Ball lemma]
$d(x)= \sum_{y \in B(x)} \chi(S^+(y))$.
\label{balllemma}
\end{lemma}
\begin{proof}
From the multiplicative property of the genus under the joint operation $S^+(x) + S^-(x) = S(x)$,
we have
$$  \sum_{y \in B(x)} (1-\chi(S^+(y))) = \sum_{y \in B(x)} \omega(y) (1-\chi(S(y))) \; . $$
Now use Gauss-Bonnet on the right hand side which tells that it is equal to $\chi(B(x))=1$.
The left hand side is $\sum_{y \in B(x)} 1$ - $\sum_{y \in B(x)} \chi(S^+(y))$ which is 
$$  1+d(x) - \sum_{y \in B(x)} \chi(S^+(y)) = 1 \; . $$
\end{proof}

\paragraph{}
{\bf Example:} if $G=K_3$ is the triangle graph and $x$ is the central vertex of 
$G_1$ which has dimension $2$,
then $d(x)=6$ and every $\chi(S^+(y))$ in $B(x)$ except $x$ itself has $\chi(S^+(y))=1$.
In a general complex $G$, If $x$ is a facet in $G$ (a face of maximal dimension),
then $\chi(S^+(x))=0, \chi(S^+(y))=1$ for all neighbors.

\paragraph{}
When looking at Theorem~(\ref{2}), it suggests to lump the potential energies $g(x,y)$ 
together and see it as a curvature. This indeed works.
The {\bf potential} $V(x) = \sum_{y} g(x,y)$ at the simplex $x$ is the sphere curvature

\begin{lemma}[Potential is curvature]
$V(x) = \sum_{y} g(x,y) = \omega(x) g(x,x) = k(x)$
\end{lemma}

\begin{proof}
The claim
$$ \sum_y g(x,y) = (-1)^{{\rm dim}(x)} g(x,x) = k(x) $$
can be restated in vector form as 
$$  g 1 = k \; , $$ 
where $1$ is the vector $1(x)=1$. As $g=(1+A)^{-1}$, this is equivalent to 
$$  L k = (1+A) k = 1  \; . $$
We show now this. Because $k(x)=1-\chi(S^+(y))$ and $A$ is the adjacency matrix of $G'$, this means
$$ L k = (1+A) (1-\chi(S^+(y))) = 1 + d_{G'}(x) - \sum_{y \in B_{G'(x)}(x)} \chi(S^+(y))  \; . $$
Using Lemma~(\ref{balllemma}) this is equal to $1$. 
\end{proof}

\section{Proof of Theorem~(\ref{2})}

\paragraph{}
The proof is inductive in the number $n$ of cells in $G$. It is again
easier to prove the result in the more general class of CW complexes.
Let $L$ be the connection matrix of $G$ and $K$ the connection matrix of $G+x$.
Define $K(t)(y,z) = K(y,z)$ if $z$ is different from $x$ and $K(t)(y,x)= t K(y,x)$ if
$y \neq x$ and similarly $K(t)(x,y)=t K(x,y)$ if $y \neq x$:
$$ K(t) = \left[ \begin{array}{ccccccc}
    L_{11} & L_{12} & . & . & . & L_{1n} & t L_{1,x}\\
    L_{21} & L_{22} & . & .  & . & . & t L_{2,x} \\
     . & . & . & . & . & . & t L_{3,x} \\
     . & . & . & . & . & . & .   \\
     . & . & . & . & . & . & .   \\
     . & . & . & . & . & . & .   \\
    L_{n1} & . & . & . & . & L_{nn} & t L_{n,x} \\
    t L_{x1} & t L_{x2} & \dots & \dots & \dots & t L_{xn}& 1  \\
\end{array} \right] . $$
Then $K(0)=L \oplus 1$ and $K(1)=K$. The matrix $K(0)$ has the eigenvalues of $L$
and an additional eigenvalue $0$.

\paragraph{}
Let $L[y,x]$ denote the minor of $L$ in which row $y$ and column $x$ were deleted. 

\begin{lemma}
$\det(L[y,x])$ is $\omega(y)$ if $y$ is a subset of $x$
and $x$ is facet, a maximal element.
\end{lemma}
\begin{proof}
The Green-Star formula implies
$$ g(x,y) = \omega(x) \omega(y) \chi(W^+(x) \cap W^+(y)) \; . $$
Now $\chi(W^+(x) \cap W^+(y)) = \chi(W^+(x))$ because of inclusion and
this is equal to $\omega(x)$ due to maximality. 
\end{proof}

This immediately implies $\sum_{y \subset x} \det(L[y,x]) = \chi(S(x))$ if $x$ is 
a facet, a set in $X$ which is not contained in a larger subset. 

\paragraph{}
The following proposition is a version of the 
multiplicative Poincar\'e-Hopf result \label{fredholmextension}. 
It will relate $\psi(G +_{A} x)$ with $\psi(G)$, where $A$ is the glue sphere at which
$x$ has been attached. 

\begin{propo}
${\rm det}(K(t)) = (1-t^2 \chi(A)) {\rm det}(L)$.
\end{propo}
\begin{proof}
A Laplace expansion with respect to the last column (which by definition
belongs to a maximal cell) gives
$$ {\rm det}(K(t)) = {\rm det}(L) - t^2 \sum_{y \subset x} {\rm det}(L) \omega(y) \; .  $$
Now use the previous lemma $\det(L[y,x]) = \omega(y)$ to get
$$ \sum_{y \subset x} {\rm det}(L) \omega(y) = {\rm det}(L) \chi(A) \; . $$
\end{proof}

\paragraph{}
As in a CW complex, $\chi(A)$ is either $0$ or $2$, the
lemma implies that in the later case, at $t=1/\sqrt{2}$ the determinant is
zero, with a single root meaning that a single eigenvalue crosses $0$.
In the former case $\chi(A)=0$, the determinant $K(t)$ stays constant meaning
that no eigenvalue can cross $0$.

\section{The strong ring}

\paragraph{}
Theorem~(\ref{1}) and Theorem~(\ref{2}) can be generalized
to a ring generated by simplicial complexes in which the disjoint union is the addition and
the Cartesian product is defined as a CW complex. 
The notion of ``Cartesian product" is pivotal for any geometry. 
The fact that the Cartesian product of simplicial complexes (as sets) is not a simplicial complex
has as a consequence that the connected simplicial complexes are 
the {\bf primes} in a ring in which additive primes - the connected elements - have 
a unique prime factorization. The product enjoys both spectral compatibility 
for the Hodge Laplacian as well as connection Laplacian. Both for the Hodge Laplacian as
well as the connection Laplacian the spectrum adds. For the connection Laplacian, the spectrum
multiplies. 

\paragraph{}
With the disjoint union of simplicial complexes as addition, the set of simplicial complexes is 
an additive monoid for which the empty complex is the zero element. One can extend it to a group, the
free Abelian group generated by connected complexes. The set theoretical {\bf Cartesian product} $G=A \times B$ of two 
simplicial complexes is not a simplicial complex any more. It can be given the structure of a discrete CW complex
and still defines two graphs $G_1$, the Barycentric refinement of $G$ and $G'$, the connection graph of $G$. 
Both graphs have the sets in $A \times B$ as vertices. In $G_1=A \times B$, two sets $x,y$ are connected, 
if $x \subset y$ or $y \subset x$. In $G'$, two different sets $x,y$ are connected if $x \cap y$ is not-empty. 
The {\bf refined Cartesian product} $A \times_1 B = (A \times_s B)_1$ is now a complex which satisfies 
all the properties of the continuum like K\"unneth. It is not associative however 
as $(A \times K_1) \times K_1$ is the second Barycentric refinement of $A$ while 
$A \times (K_1 \times K_1) = A \times K_1$ is the first Barycentric refinement. 

\paragraph{}
The disjoint union and the Cartesian product defines the {\bf ring} $\G$ generated by simplicial complexes. 
The multiplicative unit is $1=K_1$ and the empty complex $\emptyset$ is the {\bf $0$-element}. 
We call it the {\bf strong ring} because the corresponding connection graphs get 
multiplied with the {\bf strong Sabidussi multiplication} for graphs which is dual to the 
multiplication in the Zykov-Sabidussi ring. 
See \cite{StrongRing,numbersandgraphs}.
The ring elements can also be represented by graphs, the Barycentric refinement graph $G_1$ 
in which two elements are connected if one is contained in the other or the {\bf connection graph} $G'$
in which two elements are connected if they intersect. 

\paragraph{}
Every element $G = G_1 \times \cdots \times G_n$ of the ring has a 
connection Laplacian $L(G)$. Just define $L(G) = \sum_i L(G_1 \times \cdots \times G_n)$ as 
a tensor product of $(L(G_1) \otimes L(G_1) \cdots \otimes L(G_1)) \otimes (L(G_2) \cdots \otimes L(G_2))  
\cdots \otimes (L(G_n) \otimes L(G_n) \cdots \otimes L(G_n))$. 
The {\bf strong product} of two finite simple graphs $G=(V,E)$ and $H=(W,F)$
is the graph $G \osquare H = (V \times W, \{ ((a,b),(c,d)) \; a=c, (b,d) \in F \} 
                                    \cup  \{ ((a,b),(c,d)) \; b=d, (a,c) \in E \}
                                    \cup  \{ ((a,b),(c,d)) \; (a,c) \in E {\rm and} (b,d) \in F \} )$.
It is an associative product introduced by Sabidussi \cite{Sabidussi}. See \cite{ImrichKlavzar,HammackImrichKlavzar}.

\paragraph{}
This shows that the Zykov-Sabidussi ring with join as addition and large multiplication
is implemented as a strong ring of connection graphs and algebraically 
has a representation in a tensor ring of connection matrices. 

\begin{lemma}
$(G \times H)' = G' \square H'$
\end{lemma}
\begin{proof}
In both cases, we deal with a graph with vertex set $V \times W$, where
$V=G$ is the set of simplices in $G$ and $W=H$ the set of simplices in $H$. 
Two elements in $G \times H$ intersect if they either agree on one side or
if they are connected on both sides. This is exactly what the Sabidussi multiplication
does. 
\end{proof}

\paragraph{}
The energy theorem extends from abstract simplicial complexes to elements in the strong 
ring $\G$. First of all, we have to see that for any $G \in \G$, the Laplacian $L$ 
is unimodular $g=L^{-1}(x,y)$ and satisfies $\chi(G) = \sum_{x,y} g(x,y)$. 
As the proof of the unimodularity theorem worked more generally for discrete
CW complexes, the proof goes over. Also for CW complexes, 
every unit sphere can be decomposed into a stable and unstable part as a join.
The Gauss Bonnet formula for the curvature $1-\chi(S^-(x))$ is the definition 
of Euler characteristic, for $K^+(x)=1-\chi(S^+(x))$. The formula $g(x,x) = 1-\chi(S(x))$
extends in the same way. Also $\sum_y g(x,y) = g(x,x)$ generalizes and 
the energy theorem follows.

\begin{coro}
The energy functional is a ring homomorphism from $\G$ to $\ZZ$. 
\end{coro}

\paragraph{}
The Cartesian product of cell complexes produces a strong product for the
connection graphs and a tensor product for the connection matrices: 

\begin{propo}
If $G$ and $H$ are two simplicial complexes with connection matrix $L(G)$ and $L(H)$.
Then $L(G \times H) = L(G) \otimes L(H)$. 
\end{propo}
\begin{proof}
If $x_1, \dots x_n$ are the cells in $G$ and 
$y_1, \dots, y_m$ are the cells in $H$, we have
basis elements $e_1, \dots, e_n$ and 
$f_1, \dots f_m$ the basis elements in the Hilbert
space of $H$. Now build the basis 
$e_1 \otimes f_1, \dots, e_1 \otimes f_m,
 e_2 \otimes f_1, \dots, e_2 \otimes f_m,
 f_m \otimes f_1, \dots, e_n \otimes f_m$ for the
tensor product of the two Hilbert spaces. In that
basis, the connection matrix is the tensor product
of the $L(G),L(H)$. In the first row, we have
the blocks $L(G)_{11} [L(H)], \dots, L(G)_{1n} [L(H)]$.  
\end{proof} 

\paragraph{}
As stated in Theorem~\ref{representation},
it follows that the strong ring has a {\bf representation} in a 
tensor algebra of matrices. Every element $G$ in the ring is given 
a connection matrix $L$. The addition in the ring is the direct
sum $\oplus$ of matrices. The multiplication in the ring produces the
tensor product of matrices. The strong ring generated by simplicial complexes
therefore has a representation in a tensor algebra. 

\section{Green function entries}

\paragraph{}
The {\bf star} $W^+(x)$ of a simplex $x \in G$ is defined as
$W^+(x) = \{ y \in G \; | \; y \subset x \}$. It contains also $x$. 
We think of it as the {\bf unstable manifold} passing through $x$.
Unlike the {\bf core} simplicial complex $W^-(x) = \{ y \in G \; | \; y \subset x \}$, the star
is not a simplicial complex in general. The {\bf core} of $x$ is the symplicial complex generated by 
the simplex $x$.  We can write $L(x,y) = \chi(W^-(x) \cap W^-(y))$ 
and $\omega(x) = \chi(W^-(x) \cap W^+(x))$
as $W^- \cap W^+ = \{ x\}$ is a set of sets with only one element $x$ which again is not
a simplicial complex in general. 

\paragraph{}
The Green functions, the entries of the 
inverse matrix $g=L^{-1}$ have a similar formula. It generalizes
the formula $g(x,y) = \omega(y)$ if $y$ is a subset of $x$. 

\paragraph{}
The proof of the Green Star formula 
uses a simple lemma about complete complexes. We use the notation 
$x \sim z$ if $x$ and $y$ intersect. 

\begin{lemma}
If $G$ is a complete complex with maximal simplex $u$ and $x \in G$, then 
$$ \sum_{z \in G, z \sim x}  \omega(x) \omega(z) \omega(u) = \delta_{x,u} \; . $$
\end{lemma}
\begin{proof}
If $x=u$, then the statement re-reads $\sum_{z, z \sim x} \omega(x) \omega(z) = \omega(u)$
which is trivial as $\omega(x)=\omega(u)$ and $\sum_{z \in G} \omega(z)=1$ for a complete
complex. \\
if $x \neq u$, then this restarts as $\sum_{z \sim x} \omega(z) = 0$. This follows
from $\sum_{z} \omega(z)=1$ and $\sum_{z \cap x = \emptyset} \omega(z) =1$.
\end{proof}

\paragraph{}
This implies for example that $\sum_{x,z \in G, x \sim z} \omega(x) \omega(z) \omega(u)=1$
which is a restatement that the {\bf Wu characteristic} of $G$ defined by 
$\omega(G) = \sum_{x \sim y} \omega(x) \omega(y)$ agrees with $\omega(u)$, if $u$ is 
the maximal simplex in $G$. For more on Wu characteristic, see \cite{valuation,CohomologyWuCharacteristic}.

\begin{proof}
The Green star formula is equivalent to the following statements about stars. 
From matrix multiplication $L g = 1$ we have
$$ \sum_{z} L(x,z)   \omega(z) \chi( W^+(z) \cap W^+(y) )  = 0 $$
if $x \neq y$ and
$$ \sum_{z} L(x,z)   \omega(z) \chi( W^+(z) \cap W^+(x) )  = \omega(x) \; .  $$
Using the notation $z \sim x$ if $x \cap z$ intersect. The first relation means for $x \neq y$
$$ \sum_{z \sim x}  \omega(z) \chi( W^+(z) \cap W^+(y) )  = 0 $$
and the second means
$$ \sum_{z \sim x}  \omega(z) \chi( W^+(z) \cap W^+(x) )  = \omega(x) \; .  $$
We can rewrite the second statement as 
$$ \sum_{u, x \subset u} \sum_{z \subset u, z \sim x} \omega(z) \omega(x) \omega(u) = 1 \; .  $$
By the above lemma, all terms with $x \neq u$ do produce $0$ and the identity reduces to
$$ \sum_{z \subset x} \omega(z)  =  1 \;  .$$
Similarly, the first statement is for $x \neq y$
$$ \sum_{u, x \subset u} \sum_{z \subset u, z \sim y} \omega(z) \omega(x) \omega(u) = 0 \; .  $$
This follows again from the above lemma, as $u=y$ is prevented from the assumption $y \neq x$
and the lemma implies that for $u \neq x$ 
$$ \sum_{z \subset u, z \sim y} \omega(z) = 0 \; . $$
\end{proof}

\paragraph{}
The diagonal entries of $g$ are are genus of the unit spheres: 

\begin{coro}
$g(x,x) = \chi(W^+(x)) = 1-\chi(S(x))$. 
\end{coro}

\paragraph{}
If we think of $f(x)={\rm dim}(x)$ as a scalar function on the simplicial complex $G$, 
then Poincar\'e-Hopf shows that every point $x$ is a critical point with index $\omega(x)$
and that $f$ naturally is a Morse function. 
The simplicial complex $W^-(x)$ is in this picture {\bf stable manifold} of the gradient flow
and the star $W^+(x)$ the {\bf unstable manifold}. The sets $W^+(x) \cap W^+(y)$ and
$W^-(x) \cap W^-(y)$ are then {\bf heteroclinic connection points} and $W^+(x) \cap W^-(x)$ are 
{\bf homoclinic connection points}. We see that the matrix entries of $L$ and $g$
are all given by Euler characteristic values of homoclinic or heteroclinic connection sets. 
The picture indicates that every finite abstract simplicial complex defines a hyperbolic 
structure whose homoclinic and heteroclinic manifolds build the connections. 

\paragraph{}
Since $\chi(W^-(x))=1$ for all $x$ and $\omega(x) = \chi(W^0(x))$, 
an elegant symmetric description is
$$ L(x,y) = \chi(W^-(x)) \chi(W^-(y)) \chi(W^-(x) \cap W^-(y)) \; . $$
$$ g(x,y) = \chi(W^0(x)) \chi(W^0(y)) \chi(W^+(x) \cap W^+(y)) \; . $$

\section{Footnotes}

\paragraph{}
The results in this paper have appeared in previous research notes. 
This is an attempt for a self-contained publishable write-up. Theorem~(\ref{1}) was 
announced in the spring of 2016 and appeared first in \cite{Unimodularity}.
An important final step happened in November 2016 with the multiplicative 
Poincar\'e-Hopf result. 
%(See the blog entry ``The Unimodularity Theorem for CW Complexes
%from November 26, 2016 in quantumcalculus.org). 
Theorem~(\ref{2}) is in \cite{Spheregeometry}, where the set of diagonal elements
is outed as a combinatorial invariant as well as \cite{DehnSommerville}, where
the functional ${\rm tr}(L-g)$ is discussed. The arithmetic 
is discussed in \cite{ArithmeticGraphs,StrongRing}. The unimodularity theorem
Theorem~(\ref{1}) has been proven differently without the use of CW complexes in 
\cite{MukherjeeBera2018}.

\paragraph{}
The name ``energy" suggests itself from the fact that for any Laplacian $L$,
the entry $g(x,y)=V_x(y)$ is the {\bf potential energy} at $y$ if a mass
has been placed at $x$ so that the sum over all $g(x,y)$ can
be seen as the total potential theoretical energy of $G$.
For the Laplacian $L=-\Delta/(4\pi)$ on $\RR^3$,
the value $V_y(x) = g(x,y)=|x-y|^{-1}$ is
the {\bf Newton potential} of a mass point at $y$ and
$\int \int |x-y|^{-1} \; d\mu(x) d\mu(y)$ is
the total energy of a measure $\mu$. If $\mu$ is a
{\bf mass density} then this is the classical
{\bf total potential theoretic energy} of the mass configuration.
In the case of a simplicial complex, we could look also at the energy
${\rm E}(\mu) = \sum_x \sum_y g(x,y) d\mu(x) d\mu(y)$ of a
measure $\mu$ given by a measure vector
$\mu(x)_{x \in G}$ with $\mu(x) \geq 0$.
In our case, we look at the uniform measure
giving every simplex the same weight $1$.

\paragraph{}
The definition of a CW-complex could be generalized
by using ``homotopic to $K_1$" instead of contractible. But this would complicate
proofs. Examples, where contractibility does not agree with ``homotopic to $K_1$"
is the {\bf dunce hat} or the {\bf Bing house}. It has a relation with {\bf shellability},
the property that $G$ is pure of some dimension $d$ and that there is an ordering $x_1, \cdots, x_n$ of maximal
$d$-simplices such that the complex generated by $x_1,\dots x_k$ 
intersected with the complex generated by $x_{k+1}$ is a $(d-1)$-dimensional
shellable complex. An elegant generalization which avoids discrete homotopy is to build
CW-complexes in which one replaces spheres by {\bf Dehn-Sommerville complexes}, complexes for
which the $f$-function $f_G(t) = 1+f_0 t + f_1 t^2 + \cdots + f_d t^{d+1}$ has the property
that $f(t-1/2)$ is even for even $d$ and odd for odd $d$. There is a class of
Whitney complexes defined recursively as graph sets $\mathcal{X}_{-1}=\{ 0 \}$,
$\mathcal{X}_d = \{ G \; | \; \chi(G)=1+(-1)^d$ and $S(x) \in \mathcal{X}_{d-1} \}$ which
are all Dehn-Sommerville as a consequence of Gauss-Bonnet 
$f_G(t) = \sum_x F_{S(x)}(t)$,
where $F_G(t)$ is the anti-derivative of $f_G(t)$. The energy theorem generalizes to this
class of CW complexes because the multiplicative Poincar\'e-Hopf theorem works there.

\paragraph{}
The development of the notion of a manifold \cite{Scholz} is closely related
to combinatorial structures built by Van Kampen, H. Weyl \cite{Weyl1925} (page 10)
or J.H.C. Whitehead, who also happend to introduce CW complexes.
The story of CW complexes starts with Ehresmann (1933) 
and is linked to algebraic developments \cite{Dieudonne1989}. 
The papers of Whitehead \cite{Whitehead} from 1939 to 1941 established that
CW complexes have nice homotopy properties. The definition of discrete CW complexes
was done to prove unimodularity \cite{Unimodularity}.
It depends on the combinatorial definition of what a "sphere" is. Already 
van Kampen did such constructs in 1929. A subclass of CW complexes 
is the strong ring generated by simplicial complexes \cite{StrongRing}. The strong ring is
isomorphic to a subring of the Sabidussi ring \cite{Sabidussi} and can be seen as a subring of a general 
algebraic Stanley-Reisner ring construction \cite{Stanley86}. The strong ring is a category of 
objects where Gauss-Bonnet, Brouwer-Lefshetz fixed point theory \cite{brouwergraph}, 
Euler-Poincare, arithmetic compatibility with the spectrum and that the energy theorem works
and that energy is a ring homomorphism on the strong ring. 

\paragraph{}
Abstract simplicial complexes and graphs have been close to each other 
already in the context of graph coloring and topological graph theory as in \cite{TuckerGross}. 
Many mathematicians, in particular Poincar\'e, Birkhoff or Whitney worked both in discrete 
settings as well as using continuum topology. Graphs are still less used in topology, maybe because
{\it ``the origins of graph theory are humble, even frivolous"} to quote \cite{BiggsLloydWilson}. An other
reason is that graphs often are treated as {\bf $1$-dimensional simplicial complexes} and 
are  not equipped with more powerful simplicial complexes like the clique complex. 
 Ivashchenko \cite{I94} translated Whitehead's homotopy notion into concrete procedures 
in graph theory. It has been simplified in \cite{CYY} which is the version we use and which is 
crucial for defining "sphere" combinatorially.

\paragraph{}
Combinatorial definitions and characterizations of spheres were looked for already in \cite{Weyl1925}. 
One can define a ``sphere" using Morse theory as classically spheres 
are spaces on which the minimal number of
critical points of a Morse function is two. These things are mostly equivalent \cite{knillreeb}. 
The Morse approach is used in Forman's discrete Morse theory \cite{forman95} and based on the classical 
Reeb sphere theorem characterizing spheres as manifolds for which the minimal number of critical 
points is $2$. 
An important property of spheres are Dehn-Sommerville relations which imply $f(x-1/2)$ being either
even or odd allowing to look at geometries where the Dehn-Sommerville spaces replace spheres. It is natural
as both Dehn-Sommerville spaces as well as spheres are invariant under the join operation. See
\cite{dehnsommervillegaussbonnet}. 

\paragraph{}
The first definition of abstract simplicial complexes was given in 1907 by Dehn and Heegaard (see
\cite{BurdeZieschang}). First attempts to define discrete manifolds go back to Tietze in 1908. 
More work was done by Brouwer, Steinitz, Veblen, Weyl and Kneser. 
The first textbook in which abstract simplicial complexes were used
heavily and stressed is \cite{alexandroff}, a book first published in 1947. 
Alexandroff calls a simplicial complex an ``unrestricted 
skeleton complex" and an arbitrary set of finite sets is called there a {\bf skeleton complex}.
The notion of simplicial complex is still mostly used in Euclidean settings
like in \cite{Hatcher}, but abstract simplicial complexes have entered modern
topology textbooks like \cite{spanier,Maunder} of \cite{LeeTopologicalManifolds}. 
Its use was amplified with the emergence
of {\bf simplicial sets}, which generalize simplicial complexes. In \cite{Maunder}, also
the join of two simpicial complexes is defined. The analogue definition for graphs origins 
from Zykov \cite{Zykov}. Many examples simplicial complexes in graphs are covered in \cite{JonssonSimplicial}. 
It uses a notion of abstract simplicial complex in which empty sets (the "void") is present.
(This is for example used in \cite{Stanley86}).  % page 120
This leads to {\bf reduced $f$-vectors} $(f_{-1},f_0,\dots)$ with $f_{-1}=1$ and 
{\bf reduced Euler characteristic} $\chi(G)-1$ used in enumerative combinatorics \cite{Stanley86}.

\paragraph{}
Notions of {\bf shellability} which are related to homotopy came only later. 
Bruggesser and Mani \cite{BruggesserMani} in 1971 gave the first recursive combinatorial definition.
They cite Rudin's paper \cite{Rudin1958} who calls a triangulation of 
a space shellable if there exists an ordering of the maximal simplices such that during the build up, 
the complexes $G_0,G_1,\dots$ 
are all homeomorphic to $G$. Rudin cites D.E. Sanderson, a paper from 1957 who calls a $d$-manifold 
with boundary shellable if there is an ordered cellular decomposition into d-cells with disjoint interior such that $X_k$ intersected with the boundary $G_{k-1}$
is homeomorphic to a $(d-1)$-cell. Also this definition makes uses of homeomorphism and so Euclidean embeddings. 
Sanderson actually works with what we today call the PL manifolds and  shows that for any triangulation, there is as a 
shellable subdivision. The Sanderson paper cites a paper of R.H. Bing from 1951 who looks at a "partitioning" into mutually 
exclusive open sets in Euclidean space whose sum is dense. If $G_k$ is an ordered sequence of these sets, then 
$G_k$ intersected with $X_{k+1}$ is a connected set. 
The Bing house can be thickened and triangulated to be unshellable even so the thickened house is a 3-ball.

\paragraph{}
Green's functions are ubiquitous in mathematical physics. Whenever there is an
operator $H$ serving as a Laplacian, then $g(x,y)$ are the matrix entries of the inverse of $H$. The entries $g(x,x)$
usually do not exist, like for $-\Delta/(4\pi)$ on $\mathbb{R}^3$, where $g(x,y)=|x-y|^{-1}$ is the Newton potential.
One looks therefore at the inverse $g_{\lambda}(x,y)$ of $H-\lambda$, where $\lambda \in \CC$ is a 
parameter. Now, $\lambda \to g_{\lambda}(x,y)$ is analytic outside the spectrum of $H$. The diagonal Green's function 
$\lambda \to g_{\lambda}(x,x)$ is a Herglotz function if $H$ was self-adjoint. While for real $\lambda$, this might 
not exist, one can look at limits like the {\bf Krein Spectral shift} 
$\xi(x,\lambda)=\lim_{\epsilon \to 0} {\rm arg}(g_{\lambda+i\epsilon}(x,x))$ 
which exist for almost all $\lambda$ \cite{xifunction}.
If operator valued random variables are used, the limit of the
expectation of $g_{\lambda+i\epsilon}$ often exists almost everywhere and encodes the derivative of a Lyapunov exponent and its conjugate, 
the density of states. In one-dimensional Schr\"odinger setups, one has $g_{\lambda}(x,x) = 1/(m^+(x,\lambda)-m^-(x,\lambda))$, 
where $m^{\pm}$ are Oseledec spaces. The unboundedness of $g$ is reflected in {\bf non-uniform hyperbolicity} of the cocycle dynamics,
for which stable and unstable directions can get arbitrarily close. In the current 
discrete setting, where for $\lambda=0$ things stay bounded, we don't have technical difficulties
and Green's functions given by topological ``curvatures" $g(x,x)$ are bounded. We are in a {\bf uniformly hyperbolic situation}. 

\begin{comment}
<< ~/text3/graphgeometry/all.m; s = ErdoesRenyi[10, 0.4]; L = ConnectionLaplacian[s]; 
g[lambda_] := Inverse[L - lambda]; xi[lambda_] := Table[Arg[ g[lambda][[k, k]]], {k, Length[L]}]/Pi;
Show[Table[Plot[xi[t][[k]], {t, -2, 2}],{k,Length[L]}]]
\end{comment}

\paragraph{}
Classically, if we have a {\bf hyperbolic fixed point} $x$ of a transformation or a hyperbolic
{\bf equilibrium point} of a flow on a $n$-dimensional 
Riemannian manifold $M$, then a {\bf theorem of Grobman-Hartman-Sternberg} assures the existence of stable 
and unstable manifolds $W^{\pm}(x)$ passing through $x$.
If we take a small geodesic sphere $S(x)$ around $x$, then $S^{\pm}(x) = S(x) \cap W^{\pm}(x)$ 
are $(k-1)-$ and $(d-k-1)$-dimensional spheres and also classically 
$S(x) = S^+(x) + S^-(x)$, where $+$ is 
the join. We can see the decomposition of a sphere $S(x)$ into a join of of a stable and unstable part
as such a hyperbolic structure for the gradient flow of the dimension functional $f(x)={\rm dim}(x)$ 
on $G$. Every $x \in G$ is a critical point. If we look at the Morse cohomology on the Barycentric
refinement of this complex, we get the original simplicial cohomology. The fact that simplicial cohomology 
does not change under Barycentric refinement can be seen as a simple example where Morse and simplicial 
cohomology agree in a discrete setting. We must stress however that this equivalence holds in the discrete
for any simplicial complex $G$. There is no discrete manifold structure required. 

\paragraph{}
The total energy $\sum_{x,y} g(x,y)$ is the discrete analogue of classical energies in potential theory. 
In dimension $d>2$, the {\bf Newton potential} in $\mathbb{R}^d$ is $g(x,y)=|x-y|^{2-d}$. 
The prototype is the {\bf scalar Laplacian} $L=-\Delta/(4\pi)$ in $\RR^3$. The Gauss law 
$Lf=\rho$ reproduces via the divergence theorem the Newton law, the inverse square law of gravity or
electro statics. Gauss noticed that $g(x,y) = 1/|x-y|$, explaining gravity. 
But the Gauss discovery allows also to define Newton's law on any space equipped with 
a Laplacian. The two dimensions, the Laplacian $-\Delta/(2\pi)$ on $\RR^2$ gives the Green's function $g(x,y) = -\log|x-y|$ and
$-\int \int \log|x-y| \; d\mu(x) d\mu(y)$ is the logarithmic energy which can also be finite for singular measures like
measures located on Julia sets. For discrete measures $\mu$, one can disregard selfinteraction and get the 
logarithm of van der Monde determinant. The one-dimensional case $-\Delta/2 = d^2/2dx^2$ with $g(x,y) = |x-y|$ is used 
in statistics.  Its energy $I(\mu) = \int_{\RR} \int_{\RR} |x-y| d\mu(x) d\mu(y)$ is the {\bf Gini energy}.
It plays a role in statistics, where $\mu$ is the law of a random variable. 
% Integrate[ Abs[x-y] Exp[-x^2] Exp[-y^2], {x,-Infinity,Infinity},{y,-Infinity,Infinity}]
Discrete Green functions in graph theory were studied in \cite{ChungYau2000} and appear naturally in any Markov setting. 
Let us mention that for the Hodge Laplacian $H=(d+d^*)^2$, whose inverse $h$ on the ortho complement of harmonic functions
is defined (which is called the pseudo inverse), we have no interpretation of the total energy $\sum_{x,y} h(x,y)$ yet. 
There might be no natural one. 

\paragraph{}
Discrete curvatures have first appeared in graph coloring contexts \cite{Heesch} but only 
in a two-dimensional setup. See \cite{Higuchi}. The discrete Gauss-Bonnet theorem is different
from Gauss-Bonnet theorems for polyhedra where excess angles matter  (see e.g. \cite{Polya54}). An attempt to produce a 
discrete second order curvature worked in the simplest Hopf Umlaufsatz situation \cite{elemente11}.
After working on this more \cite{poincarehopf,cherngaussbonnet} and integral theoretic results
\cite{indexexpectation,colorcurvature} (where the later are discrete analogues of Banchoff theorems
\cite{Banchoff1967}) the situation is transparent: any deterministic or random diffusion 
process can be used to disperse the original curvatures $\omega(x)$ on simplices to the zero dimensional
part. If done in the most symmetric way to the nearest zero dimensional point one gets the Levitt
curvature $K(x)$ \cite{Levitt1992}. If done along the gradient of a function $f$ one gets Poincar\'e-Hopf. 
An example was $f=-{\rm dim}$, where the curvature is $k(x)=\omega(x) (1-\chi(S(x)))$ we have seen here.
If the diffusion is done on the right scale together with adapted Barycentric refinements,
one expects to get the Euler curvature in a Riemannian manifold limit. 

\paragraph{}
The energy theorem is part of a story in arithmetic:
there is a ring $\mathcal{Q}$ of finite simple graphs which after taking graph complements 
becomes a ring $\overline{\mathcal{Q}}$. Now, if we look at the strong ring generated by 
simplicial complexes, the corresponding connection graphs define a subring $\overline{\mathcal{Q}}$ 
on which Euler characteristic is by the energy theorem given as a sum of matrix inverse elements. The ring $\mathcal{Q}$
has a representation in a matrix tensor ring. Every element in $\mathcal{Q}$ is represented
by a matrix, the connection Laplacian of its graph complement. We can so attach a spectrum
$\sigma(G)$ of a ring element in $\mathcal{Q}$. Under addition, the spectra add, under 
multiplication, the spectra multiply. The same happens with Euler characteristic. 
The upshot is that every generalized rational number,
an element $x$ in $\mathcal{Q}$ defines a particle system to which we can attach an Euler characteristic
$\chi(\overline{x})$ which is a total potential energy of the complex $\overline{x}$ and also attach a matrix 
$L(x)$ whose spectrum is compatible with the ring operation. 

\paragraph{}
There are many open questions. We would like to know about  energy moment 
$\sum_{x,y} g(x,y)^3$. It appears often to be an integer multiple of the energy 
$\sum_{x,y} g(x,y)$. While we know $g(x,x) = 1-\chi(S(x))$ and that the set of
these diagonal entries are a combinatorial invariant in the sense that the set is stable already after one
Barycentric refinement of a complex we don't have an interpretation
for $k$-moments $\sum_{x,y} g(x,y)^k$ of the energy yet. Especially interesting looks the variance
$\sum_{x,y} (g(x,y)-m)^2/n^2$, where $m=\chi(G)/n^2$ is the expectation of the matrix entries of $g$. 
It is a measure for the {\bf energy fluctuation} in the complex. 
Having developed parametrized versions of Gauss-Bonnet \cite{dehnsommervillegaussbonnet} and 
Poincar\'e-Hopf \cite{parametrizedpoincarehopf} more recently, there is also
a parametrized version of the energy theorem in which the Euler characteristic 
$f_0-f_1+f_2-\dots$ 
is replaced with $f$-functions $f(t) =1-f_0 t + f_1 t^2 - \cdots$. 
The energy theorem then tells that the sum of the matrix entries
of the inverse $g$ is the $f$ function of $G$. 

\paragraph{}
An inverse spectral problem is to hear the Euler characteristic of $G$ from the
spectrum $\sigma(H)$. We use $H=(d+d^*)^2=d d^* + d^* d$, where
$d: \Lambda^k \to \Lambda^{k+1}$ is the
exterior derivative, where $\Lambda^k$ is the $f_k$-dimensional space of discrete $k$-forms,
functions on $k$-dimensional oriented simplices of $G$. The matrices $H$ and $L$ are both $n \times
n$ matrices if $G$ has $n$ elements. If $G$ is one-dimensional, then
$L-L^{-1}$ is similar to $H$ \cite{ListeningCohomology}. This even holds for products of 
one-dimensional spaces.  One can even hear the
Betti numbers of a Barycentric refined $G$ from $H$ or $L$ as $b_0$ is the number of eigenvalues $1$ and
$b_1$ is the number of eigenvalues $-1$ of $L$. In dimension $2$ or higher,
one can not hear the Betti numbers $b_k$ of $G$ from the spectrum $\sigma(L)$ of $L$ in
general but it is conceivable that this could be true for Barycentric refinements.

\bibliographystyle{plain}

\end{document}